\documentclass[letterpaper,11pt]{amsart}
\usepackage{amscd,amssymb,amsfonts,amsmath,amsthm,amscd,mathrsfs}
\usepackage[colorlinks]{hyperref}

\def \Ql {\overline{\mathbb Q}_{\ell}}
\def \F {\mathcal F}
\def \G {\mathcal G}
\def \L {\mathcal L}
\def \Fq {\overline{\mathbb F}_{q_0}}
\def \Frob {\operatorname{Frob}}
\def \tr {\operatorname{tr}}
\def \Gal {\operatorname{Gal}}

\newtheorem{theorem}{Theorem}[section]
\newtheorem{lemma}[theorem]{Lemma}
\newtheorem{prop}[theorem]{Proposition}
\newtheorem{cor}[theorem]{Corollary}
\theoremstyle{remark}

\newtheorem{example}[theorem]{Example}
\theoremstyle{definition}
\newtheorem{defi}[theorem]{Definition}

\begin{document}

\title[L-functions of twists by Dirichlet characters]{The equidistribution of L-functions of twists by Witt vector Dirichlet characters over function fields}
\author{Will Sawin}

\maketitle

\begin{abstract} Katz showed that the $L$-functions of all Dirichlet characters of $\mathbb F_q(t)$, with conductor a fixed power of a degree one prime (or $\infty$), are equidistributed in the $q \to \infty$ limit. We generalize this statement to the $L$-functions of twists of an arbitrary Galois representation by Dirichlet characters, including independence of the $L$-functions of twists of different representations by the same Dirichlet character. A similar generalization, without the independence statement, for characters with squarefree conductor, was proven by Hall, Keating, and Roditty-Gershon. \end{abstract}
\setcounter{tocdepth}{1}

\tableofcontents

\section{Introduction and statement of results}

We state our results in three equivalent languages - automorphic forms, Galois representations, and middle extension sheaves. The language of middle extension sheaves is most conducive to proving results of this type. The equivalence between the Galois representation and middle extension languages is straightforward, while the equivalence with the automorphic language is a deep theorem of Lafforgue. Afterwards, we will make some comments on the results, their method of proof, and their relationship to past work.

Let $\mathbb F_{q_0}$ be a finite field, $\mathbb F_q$ a finite field extension, and $d \geq 1$ a natural number. Let $W_{d,q} =\left( \mathbb F_q[t]/t^{d+1} \right)^\times/ 
\mathbb F_q^\times$. There exists (Lemma \ref{cft}) an isomorphism between $W_{d,q}$ and the ray class group of $\mathbb F_q(x)$ with conductor $\infty^{d+1}$ modulo the prime $(x)$. If we add the condition that this isomorphism sends the unit group $\mathbb F_q[[ x^{-1} ]]^\times$ of the local field at $\infty$ to $\left( \mathbb F_q[t]/t^{d+1} \right)^\times $ by the natural projection after setting $t= x^{-1}$, it becomes unique. Using this, we will view characters $\Lambda$ of $W_{d,q}$ as characters of that ray class group, and, when appropriate, as Galois characters.

For a tuple $(N_1,\dots, N_r)$ of natural numbers, and a continuous, conjugacy-invariant function $f$ on $\prod_{i=1}^r U (N_i)$, let $\langle f \rangle ( c_1,\dots,c_r)$ for $c_1\dots,c_r \in U(1)$ be the integral of $f$ over the subset of $\prod_{i=1}^r U_{N_i}$ consisting of matrices with determinants $c_1,\dots,c_r$ against the unique $\prod_{i=1}^r SU(N_i)$ -invariant measure on that subset with total mass one. For instance, if $f$ is the character of an irreducible representation of $\prod_{i=1}^r U (N_i)$, then $\langle f \rangle = f$ if $f$ is one-dimensional and $\langle f\rangle=0$ otherwise.

In the language of automorphic forms, this result is as follows:

\begin{theorem}\label{main-automorphic} Let $r$ be a natural number. For all $i$ from $1$ to $r$, let $m_i$ be a natural number and let $\pi_i$ be a cuspidal automorphic representation of $GL_{m_i}( \mathbb F_{q_0}(x))$, with central character of finite order. Assume that no $\pi_i$ is isomorphic to any $\pi_j$ after twisting by any character of the form $g \mapsto \alpha^{ \deg ( \det g)}$ for any $\alpha \in \mathbb C^\times$ unless $i=j$ and $\alpha=1$. Let $\pi_{i, \mathbb F_q(x)}$ be the base change of $\pi_i$ to $GL_{m_i}( \mathbb F_q(x))$.

Then there exist natural numbers $N_1,\dots, N_r$ and for all $q$, a set $S_q$ of characters $\Lambda:  W_{d,q}\to \mathbb C^\times$, of size $(1- O(1/q)) q^{d}$, such that for all $\Lambda \in S_q$, for all $i$ from $1$ to $r$, $ L( \pi_{i,  \mathbb F_q(x)}  \otimes \Lambda, s) = \det ( 1 - \varphi_{\Lambda,i} q^{1/2 -s} )$ for some $\varphi_{\Lambda,i} \in U (N_i)$ (uniquely determined up to conjugacy by this condition). 

Assume in addition that $d \geq 4$. Let $f$ be a continuous, conjugacy-invariant function on $\prod_{i=1}^r U(N_i)$. 
Then 
\[  \lim_{q \to \infty} \left( \frac{ \sum_{\Lambda \in S_q } f ( \varphi_{\Lambda,1},\dots, \varphi_{\Lambda,r}) } { \left| S_q \right|} - \frac{ \sum_{\Lambda \in S_q } \langle f \rangle ( \det \varphi_{\Lambda,1},\dots, \det \varphi_{\Lambda,r}) } { \left| S_q \right|} \right) = 0. \]

\end{theorem}

Because the determinant of $\varphi_{\Lambda,i}$ is the constant in the functional equation of $ L (s,\pi_i \otimes \Lambda) $, this reduces the problem of studying the distribution of the $L$-functions to the problem of studying the distribution of the constants of their functional equations. This should be possible by writing the constants explicitly using Gauss sums.

In the language of Galois representations: 

\begin{theorem}\label{main-arithmetic} Let $r$ be a natural number, $\ell$ a prime not dividing $q_0$, and $\iota: \overline{\mathbb Q}_\ell \to \mathbb C$ an embedding.  For all $i$ from $1$ to $r$, let $m_i$ be a natural number and let $\rho_i$ be a continous $\ell$-adic representation of $\Gal ( \mathbb F_{q_0}(x))$, $\iota$-pure of weight $w_i$. Assume, after restricting the $\rho_i$ to $\Gal ( \overline{\mathbb F}_{q_0} (x))$, that the $\rho_i$ are irreducible and that no $\rho_i$ is isomorphic to $\rho_j$ unless $i=j$.  View characters $\Lambda:   W_{d,q} \to \mathbb \Ql^\times$ as characters of the abelianization of the Galois group of $\mathbb F_q(x)$ by the Artin map.

Then there exist natural numbers $N_1,\dots, N_r$ and a set $S_q$ of characters $\Lambda: W_{d,q}\to \mathbb  \overline{\mathbb Q}_\ell^\times$, of size $(1- O(1/q)) q^d$ , such that for all $\Lambda \in S_q$, for all $i$ from $1$ to $r$, $ \iota( L(\mathbb F_q(x),  \rho_i \otimes \Lambda, s)) = \det ( 1 - \varphi_{\Lambda,i} q^{ \frac{w_i+1}{2} -s} )$ for some $\varphi_{\Lambda,i} \in U (N_i)$ (uniquely determined up to conjugacy by this condition). 

Assume in addition that $d \geq 4$. Let $f$ be a continuous, conjugacy-invariant function on $\prod_{i=1}^r U(N_i)$. Then 
\[  \lim_{q \to \infty} \left( \frac{ \sum_{\Lambda \in S_q } f ( \varphi_{\Lambda,1},\dots, \varphi_{\Lambda,r}) } { \left| S_q \right|} - \frac{ \sum_{\Lambda \in S_q } \langle f \rangle ( \det \varphi_{\Lambda,1},\dots, \det \varphi_{\Lambda,r}) } { \left| S_q \right|} \right) = 0. \]
\end{theorem}

In the language of sheaves:

\begin{theorem}\label{main-geometric} Let $r$ be a natural number, $\ell$ a prime not dividing $q_0$, and $\iota: \overline{\mathbb Q}_\ell \to \mathbb C$ an embedding.  For all $i$ from $1$ to $r$, let $m_i$ be a natural number and let $\mathcal F_i$ be a middle extension $\overline{\mathbb Q}_\ell$-sheaf on $\mathbb A^1_{\mathbb F_{q_0}}$, $\iota$-pure of weight $w_i$. Assume that $\mathcal F_i$ are all geometrically irreducible and that $\mathcal F_i$ is not geometrically isomorphic to $\F_j$ unless $i=j$.  View characters $\Lambda:   W_{d,q}\to \mathbb \overline{\mathbb Q}_\ell^\times$ as characters of the abelianization of the Galois group of $\mathbb F_q(x)$, unramified away from infinitity, by the Artin map, and let $\mathcal L_\Lambda$ be the associated lisse sheaves on $\mathbb A^1$. 

Then there exist natural numbers $N_1,\dots, N_r$ and a set $S_q$ of characters $\Lambda:   W_{d,q}\to \mathbb \overline{\mathbb Q}_\ell^\times$, of size $(1- O(1/q)) q^d$ , such that for all $\Lambda \in S_q$, for all $i$ from $1$ to $r$, $H^1_c ( \mathbb A^1_{\overline{\mathbb F}_q}, \mathcal F_i \otimes \mathcal L_\Lambda)$ is $N_i$-dimensional and pure of $\iota$ weight $w_i+1$, all other compactly supported cohomology groups of this sheaf vanish, and $j_!(\mathcal F_i \otimes \mathcal L_\lambda) =j_*(\mathcal F_i \otimes \mathcal L_\lambda)  $ for $j: \mathbb A^1 \to \mathbb P^1$ the open immersion. 

Let $\varphi_{\Lambda,i}$ be the semisimplified conjugacy class of $\Frob_q / q^{ \frac{w_i+1}{2}}$ acting on $H^1_c ( \mathbb A^1_{\overline{\mathbb F}_q}, \mathcal F_i \otimes \mathcal L_\Lambda)$, mapped from $GL_{N_i}( \overline{\mathbb Q}_\ell)$ to $GL_{N_i}(\mathbb C)$ by $\iota$, and conjugated to lie in $U(N_i)$.

Assume in addition that $d \geq 4$. Let $f$ be a continuous, conjugacy-invariant function on $\prod_{i=1}^r U(N_i)$. Then 
\[  \lim_{q \to \infty} \left( \frac{ \sum_{\Lambda \in S_q } f ( \varphi_{\Lambda,1},\dots, \varphi_{\Lambda,r}) } { \left| S_q \right|} - \frac{ \sum_{\Lambda \in S_q } \langle f \rangle ( \det \varphi_{\Lambda,1},\dots, \det \varphi_{\Lambda,r}) } { \left| S_q \right|} \right) = 0 . \]

\end{theorem}

Our argument follows closely the method of Katz in \cite{wvqkr}. However, at several points we need different arguments to handle the additional generality, and we often find it convenient to use different notation. Thus we will mostly be unable to reuse lemmas from \cite{wvqkr}, although we do reuse many key ideas.

These theorems rely on the classification of finite simple groups, via \cite[Theorem 1.4]{guralnicktiep}.

A similar statement was proven by Hall, Keating, and Roditty-Gershon \cite[Theorems 5.0.8 and 7.0.1]{hkr}, in a different setting. Rather than working with Dirichlet characters modulo a high power of a degree one prime, they work with Dirichlet characters modulo a squarefree divisor. Because of this, there is no logical relationship between the results, and the proofs are somewhat different, as our proof follows the method of \cite{wvqkr} and theirs follows \cite{pdc}. We would like to mention some advantages of our method:  We can give independence results for the joint distribution of the $L$-functions of multiple representations twisted by the same character, our result applies without any local conditions on the representation, and least importantly, the minimum conductor of the character in our case is smaller, and is independent of the rank of the representation.

As a consequence of Theorem \ref{main-automorphic} it is possible to describe the distribution of $L( \pi \otimes \Lambda,s)$ for many Eisenstein series $\pi$, by decomposing them into cusp forms. Using this, and the methods of Keating, Rodgers, Roditty-Gershon, and Rudnick \cite{krrr}, it should be possible to estimate the variance in short intervals of the ``Fourier coefficients" of these Eisenstein series. The main new phenomenon that occurs is that the sum of the Eisenstein series against a short interval character can be expressed in terms of the trace of the Frobenius conjugacy class acting on a representation that does not factor through a product of projective unitary groups, so equidistribution and independence statements in the projective unitary group are not sufficient. Applying our theorems requires calculating the average $\langle f \rangle$ of the trace of this representation, which is equivalent by Fourier analysis to computing the average of the trace of this representation multiplied by integer powers of the determinant. The evaluation of this average should be similar to the evaluation of the average trace of a representation in \cite{krrr}, but unbalanced, and should lead to a lower-order term. As another application, it should be possible to combine this result with the methods of \cite{hkr} to calculate the variance of sums in short intervals of the generalized von Mangoldt functions associated to irreducible Galois representations, and to calculate the covariance of sums in short intervals of two different generalized von Mangoldt functions.

It is likely possible to prove these results with some degree of uniformity in the choice of sheaves, allowing for applications ``over $\mathbb Z$" (as in \cite{hkr}), but we do not pursue that here. Because we do not have uniformity, it is necessary to work with the asymptotics as $q$ goes to $\infty$ of the base change from $\mathbb F_{q_0}$ to $\mathbb F_q$ of a fixed object.

I would like to thank Chris Hall, Peter Humphries, and Edva Roditty-Gershon for helpful conversations on the topic of this paper and Ofir Gorodetsky, Emmanuel Kowalski, and Zeev Rudnick for helpful comments on an earlier draft.  The author was supported by Dr. Max R\"{o}ssler, the Walter Haefner Foundation and the ETH Z\"urich Foundation.

\section{Preliminaries}

\subsection{Class field theory}

 Let $\mathbb F_q$ be a finite field. Let $F = \mathbb F_q(x)$, let $\mathbb A_F^\times$ be the ideles of $F$, and let $K( \infty^{d+1})$ be the subgroup of $\mathbb A_F^\times$ consisting of elements which are units at each place and are congruent to $1$ modulo the $d+1$st power of the maximal ideal at $\infty$. Let $\Frob_0$ be the adele which is $x$ at the place $0$ and $1$ at every other place.

\begin{lemma}\label{cft} There exists a unique isomorphism $\mathbb A_F^\times  / F^\times K(\infty^{d+1}) \langle \Frob_0 \rangle=\left(  \mathbb F_q[t]/t^{d+1} \right)^\times/ \mathbb F_q^\times = W_{d,q} $ that restricted to the group of units $\mathbb F_q[[t]]^\times$ at $\infty$ (where we fix a uniformizer $t= x^{-1}$) is the natural projection $\mathbb F_q[[t]]^\times \to \left(  \mathbb F_q[t]/t^{d+1} \right)^\times/ \mathbb F_q^\times $. \end{lemma}

\begin{proof} Any element of $\mathbb A_F^\times$ has finitely many zeroes and poles. There is an element of $F^\times$, unique up to multiplication by $\mathbb F_q^\times$, which has exactly the same zeroes and poles, except at the place $0$.  Hence every class is $\mathbb A_F^\times / F^\times $ has a representative which is a unit at each place but $0$, unique up to multiplication by $\mathbb F_q^\times$. Thus every class in $\mathbb A_F^\times / F^\times \langle \Frob_0 \rangle$ has a representative which is a unit at each place, unique up to multiplication by $\mathbb F_q^\times$.

The group of ideles which are units at each place modulo $K(\infty^{d+1})$ is naturally isomorphic to the group of units $\mathbb F_q[[t]]^\times$ at $\infty$ modulo the units congruent to $1$ modulo $t^{d+1}$, and this quotient is $\left( \mathbb F_q[t]/t^{d+1}\right)^\times$. Modding out by $\mathbb F_q^\times$, we obtain an isomorphism, which because each class contains a representative in the unit group at $\infty$, must be unique. 

\end{proof}

\begin{lemma}\label{cft-galois}  The quotient $\mathbb A_F^\times  / F^\times K(\infty^{d+1}) \langle \Frob_0 \rangle$ is the Galois group of the maximal abelian extension of $\mathbb F_q(t)$, unramified away from $\infty$, with conductor exponent at most $d+1$ at $\infty$, and split over $0$.  \end{lemma}

\begin{proof} This is simply a description of the Artin map in class field theory. \end{proof} 

It follows that $W_{d,q}$ is a quotient of $\pi_1(\mathbb A^1_{\mathbb F_q})$, as $\pi_1$ is equal to the Galois group of the maximal extension unramified away from $\infty$. Hence any character $\Lambda: W_{d,q} \to \Ql^\times$ defines a character of $\pi_1(\mathbb A^1_{\mathbb F_q}) \to \Ql^\times$. Let let $\L_\Lambda$  (called $\L_{\Lambda(1-tX)}$ in \cite[pp. 1-2]{wvqkr}) be the unique rank one lisse sheaf whose monodromy representation is this character.

\begin{lemma}\label{norm formula} For $v$ a closed point in $\mathbb A^1_{\mathbb F_q}$ defined by an irreducible polynomial $f(x)$, $\Frob_v$ acts on $\L_\Lambda$ with eigenvalue $\Lambda ( f(t^{-1}) t^{\deg f})$.

In particular, for a point $x \in \mathbb F_q = \mathbb A^1(\mathbb F_q)$, $\Frob_{q,x}$ acts on $\L_\Lambda$ with eigenvalue $\Lambda( 1-xt)$. \end{lemma}

\begin{proof} We have defined in Lemma \ref{cft} an isomorphism of $W_{d,q}$ with a certain ray class group, that in Lemma \ref{cft-galois} we wrote explicitly as a Galois group, or as a quotient of the \'{e}tale fundamental group. By construction, the eigenvalue of $\Frob_v$ of $\mathcal L_\Lambda$ is the character $\Lambda$ applied to the element $\Frob_v$ via this pair of isomorphisms.  Thus to check this identity, it suffices to prove that this pair of isomorphisms sends the Frobenius element $\Frob_v$  to $f(t^{-1} ) t^{\deg f}$. 

The class field theory isomorphism between finite quotients of the abelianized Galois group and the idele class group sends $\Frob_v$ to the idele which is $1$ away from $v$ and the inverse of the uniformizer at $v$. We are working in the quotient of the idele class group modulo $F^\times K(\infty^{d+1})$, so we may freely multiply by the diagonal idele $x^{ - \deg f} f(x) \in \mathbb F_q(x)^\times$, obtaining an idele which is in the group of units at each finite place except $0$, has a pole of order $\deg f$ at $0$, and is equal to $f(t^{-1}) t^{\deg f} $ at $\infty$.  We are also working modulo the idele $\Frob_0$, which is $x$ at $0$ and $1$ at every other place, so we may multiply by $x^d$ at $0$, producing an idele which is a unit at each finite place and $f(t^{-1}) t^{\deg f} $. The units at each finite place lie in  $K(\infty^{d+1})$, so our original idele is equivalent to the idele which is $1$ away from $\infty$ and  $f(t^{-1}) t^{\deg f} $. Hence by the defining property of the isomorphism with $W_{d,q}$ from Lemma \ref{cft}, this idele is sent to $f(t^{-1}) t^{\deg f}$ by the isomorphism.

For the special case of degree one points, observe that the  point $x_0 \in \mathbb F_q = \mathbb A^1(\mathbb F_q)$ is defined by the irreducible polynomial $x-x_0$, so the action of $\Frob_{x_0}$ on $\L_{\Lambda}$ is given by $\Lambda ( (t^{-1}- x_0) t) = \Lambda(1-x_0t)$.\end{proof}

\subsection{Moment calculations}

For clarity, we perform the key numerical calculation that will lead to our equidistribution result in the setting of a general sequence of functions, not restricting ourselves to sheaves:

\begin{lemma}\label{diophantinepart} Let $F_q(x)$ be a sequence of functions $F_q: \mathbb F_q \to \mathbb C$ indexed by powers $q$ of $q_0$ such that $F_q(x) = O( q^{w/2})$ for all $q$ and $\lim_{q \to \infty} \frac{\sum_{x\in \mathbb F_q}|F_q(x)|^2}{ q^{w+1}}=1$. If $k\leq d$, then

\[\lim_{q \to \infty} \frac{ 1}{ q^d}  \sum_{ \Lambda: W_{d,q} \to \mathbb C^\times} \frac{ \left| \sum_{x \in \mathbb F_q} F_q(x) \Lambda(1- tx) \right|^{2k}}{q^{k(w+1)}} =   k!\]

\end{lemma}

\begin{proof} 

 Expanding and exchanging the order of summation, this is 
\[\frac{1}{ q^{k(w+1)}} \sum_{\substack{ x_1,\dots,x_k \in \mathbb F_q \\  y_1,\dots,y_k \in \mathbb F_q} }  \frac{ 1}{ q^d}  \left(\prod_{i=1}^kF_q(x_i) \overline{F_q(y_i)}\right) \sum_{ \Lambda}   \Lambda \left(\prod_{i=1}^k (1 - t x_i) (1-t y_i)^{-1}\right)\]

Now
\[ \frac{ 1}{ q^d}  \sum_{ \Lambda}  \Lambda \left(\prod_{i=1}^k (1 - t x_i) (1-t y_i)^{-1} \right)= \begin{cases} 1 & \mbox{if } \prod_{i=1}^k (1 - t x_i) \equiv \prod_{i=1}^k (1-ty_i) \pmod{ t^{d+1}} \\ 0  & \mbox{if } \prod_{i=1}^k (1 - t x_i) \not \equiv \prod_{i=1}^k (1-ty_i) \pmod{ t^{d+1}}\end{cases} \]

But because $ \prod_{i=1}^k (1 - t x_i) $ and $ \prod_{i=1}^k (1-ty_i)$ have degree $k\leq d$, this congruence holds if and only if $ \prod_{i=1}^k (1 - t x_i) = \prod_{i=1}^k (1-ty_i)$, which is true if and only if there is some permutation $\sigma$ of $\{1,\dots,k\}$ such that $y_i= x_{\sigma(i)}$. So we obtain:
\[\frac{1}{ q^{k(w+1)}} \sum_{\substack{ x_1,\dots,x_k \in \mathbb F_q \\ y_1,\dots,y_k \in \mathbb F_q} }   \prod_{i=1}^kF_q(x_i) \overline{F_q(y_i)} 1_{\exists \sigma \in S_k : y_i=x_{\sigma(i)}}\]
which equals 
\[\frac{1}{ q^{k(w+1)}} \sum_{\substack{ x_1,\dots,x_k \in \mathbb F_q \\ y_1,\dots,y_k \in \mathbb F_q\\ \sigma \in S_k \\ \forall i, y_i=x_{\sigma(i)}}}   \prod_{i=1}^kF_q(x_i) \overline{F_q(y_i)} \]
minus the contribution coming from the $x_1,\dots,x_k,y_1,\dots,y_k$ where there is more than one such $\sigma$. The contribution from each such tuple is at most $k!-1$, the number of additional $\sigma$s, times $n^{2k} q^{kw}$, the maximum value of  $\prod_{i=1}^kF_q(x_i) \overline{F_q(y_i)}$, and each tuple with more than one $\sigma$ must have $x_i=x_j$ for some distinct $i,j$, and has $y_i$ determined by $x_{\sigma(i)}$, so the space of such tuples has dimension $\leq k-1$ and thus there are $O(q^{k-1})$ such tuples, for a total contribution of $O_{n,k} ( q^{kw+k-1})$, which when divided by $q^{k(w+1)}$ is $o(1)$. So it is sufficient to estimate the limit of:
\[\frac{1}{ q^{k(w+1)}} \sum_{\substack{ x_1,\dots,x_k \in \mathbb F_q \\ y_1,\dots,y_k \in \mathbb F_q\\ \sigma \in S_k \\ \forall i, y_i=x_{\sigma(i)}}}  \prod_{i=1}^kF_q(x_i) \overline{F_q(y_i)} \]
\[ = \frac{1}{ q^{k(w+1)}}  \sum_{\sigma \in S_k}  \sum_{x_1\dots, x_k} \prod_{i=1}^k F_q(x_i) \overline{F_q(x_{\sigma(i)})}  =  \frac{1}{ q^{k(w+1)}}  \sum_{\sigma \in S_k}  \sum_{x_1\dots, x_k} \prod_{i=1}^k |F_q(x_i) |^2 \]
\[= \frac{1}{ q^{k(w+1)}}  k!  \left( \sum_x |F_q(x)|^2 \right)^k = k! \left( \frac{\sum_x |F_q(x)|^2}{ q^{w+1}}\right)^k\]
whose limit is $k!$ by the assumption that the limit of $\frac{\sum_x |F_q(x)|^2}{ q^{w+1}}$ is $1$.

\end{proof}

\begin{lemma}\label{indio} Let $F_{1,q}(x)$ and $F_{2,q}(x)$ be two sequences of functions $F_q: \mathbb F_q \to \mathbb C$ indexed by powers $q$ of $q_0$ such that for $i=1,2$ we have $F_{i,q}(x) = O( q^{w_i/2})$ and $\lim_{q \to \infty} \frac{\sum_{x\in \mathbb F_q}|F_{i,q}(x)|^2}{ q^{w_i+1}}=1$. Assume that $\lim_{q \to \infty} \frac{\sum_{x\in \mathbb F_q} F_{1,q}(x)\overline{F_{2,q}(x)}}{ q^{(w_1+w_2)/2+1}}=0$.

If $d \geq 2$, then

\[\lim_{q \to \infty} \frac{ 1}{ q^d}  \sum_{ \Lambda: W_{d, q} \to \mathbb C^\times} \frac{ \left| \sum_{x \in \mathbb F_q} F_{1,q}(x) \Lambda(1- tx) \right|^{2} \left| \sum_{x \in \mathbb F_q} F_{2,q}(x) \Lambda(1- tx) \right|^{2}}{q^{w_1+w_2+2}} =   1.\]

\end{lemma}

\begin{proof}

 Expanding and exchanging the order of summation, we must compute the limit of
\[\frac{1}{ q^{w_1+w_2+2}} \sum_{x_1,x_2,y_1,y_2 \in \mathbb F_q}  \frac{ 1}{ q^d}  \sum_{ \Lambda}  F_{1,q}(x_1) \overline{F_{1,q}(y_1)} F_{2,q}(x_2) \overline{F_{2,q}(y_2)} \Lambda \left( \frac{(1 - t x_1) (1-t x_2) }{(1-t y_1)(1-t y_2)}\right)\]

Now
\[ \frac{ 1}{ q^d}  \sum_{ \Lambda}  \Lambda  \left( \frac{(1 - t x_1) (1-t x_2)}{ (1-t y_1)(1-t y_2)}  \right)= \begin{cases} 1 & \mbox{if } \prod_{i=1}^k (1 - t x_1)(1-t x_2) \equiv (1-t y_1) (1-t y_2) \pmod{ t^{d+1}} \\ 0  & \mbox{if } \prod_{i=1}^k (1 - t x_1)(1-t x_2) \not \equiv (1-t y_1) (1-t y_2) \pmod{ t^{d+1}}\end{cases} \]

But because $(1 - t x_1)(1-t x_2)$ and $(1-t y_1) (1-t y_2)$ have degree $2\leq d$, this congruence holds if and only if $(1-t x_1) (1-tx_2) = (1-t y_1) (1-t y_2)$, which is true if and only if $x_1 = y_1$ and $x_2=y_2$ or $x_1=y_2$ and $x_2=y_1$. So the sum is equal to:
\[\frac{1}{ q^{w_1+w_2+2}} \sum_{x_1,x_2,y_1,y_2 \in \mathbb F_q}   F_{1,q}(x_1) \overline{F_{1,q}(y_1)} F_{2,q}(x_2) \overline{F_{2,q}(y_2)} 1_{(x_1 =x_2 \textrm{ and } y_1 =y_2) \textrm{ or } (x_1=y_2 \textrm{ and } x_2=y_1) }\]
We can write this as the sum over terms where $x_1=x_2$ and $y_1=y_2$, plus the sum over terms where $x_1=y_2$ and $x_2=y_1$, minus the sum over terms where both $x_1=x_2$ and $y_2=y_1$. The first sum is \[ \frac{1}{ q^{w_1+w_2+2}}\sum_{x_1,x_2\in \mathbb F_q}  F_{1,q} (x_1) \overline{F_{1,q}(x_1)} F_{2,q}(x_2) \overline{F_{2,q}(x_2)} = \frac{ \sum_{x \in \mathbb F_q} |F_{1,q}(x)|^2}{ q ^{w_1+1}}\frac{ \sum_{x \in \mathbb F_q} |F_{2,q}(x)|^2}{ q ^{w_2+1}}\] and converges to $1$ as $q$ goes to $\infty$, the second sum is \[ \frac{1}{ q^{w_1+w_2+2}}\sum_{x_1,x_2\in \mathbb F_q }   F_{1,q}(x_1) \overline{F_{1,q}(x_1)} F_{2,q}(x_2) \overline{F_{2,q}(x_2)}= \left| \frac{ \sum_{ x\in \mathbb F_q} F_{1,q}(x) \overline{F_{2,q}(x)}}{q^{(w_1+w_2)/2+1}} \right|^2\] and converges to $0$ as $q$ goes to $\infty$, and the last term is \[ \frac{1}{ q^{w_1+w_2+2}}\sum_{x \in \mathbb F_q} |F_{1,q}(x)|^2 |F_{2,q}(x)|^2= \frac{1}{ q^{w_1+w_2+2}}\sum_{x \in \mathbb F_q}  O( q^{w_1+ w_2} )= \frac{1}{ q^{w_1+w_2+2}} O (q^{w_1+w_2+1})\] and converges to $0$ as $q$ goes to $\infty$.

Hence the whole sum converges to $1$ as $q$ goes to $\infty$, as desired.
\end{proof}

\subsection{Geometric preliminaries}

Let $\ell$ be a prime and $\iota: \Ql \to \mathbb C$ an embedding.

Let $\F$ be a geometrically irreducible middle extension $\Ql$-sheaf on $\mathbb A^1$, $\iota$-pure of weight $w$.

\begin{lemma}\label{middle} Let $j: \mathbb A^1 \to \mathbb P^1$ be the open immersion. Then $j_!(\F \otimes \L_\Lambda)$ is a middle extension sheaf on $\mathbb P^1_{\mathbb F_q}$ for all but at most $m$ values of $\Lambda$. \end{lemma}

\begin{proof} 
As $ j_! (\F\otimes \mathcal L_\Lambda)$ is a middle extension sheaf away from $\infty$, it is sufficient to check that it is a middle extension sheaf at $\infty$, which occurs if and only if the invariant subspace of the local monodromy representation of $\F \otimes \mathcal L_\Lambda$ at $\infty$ vanishes. If the invariant subspace of the local monodromy representation of $\F \otimes \L_\lambda$ is nontrivial, then the local monodromy representation of $\mathcal L_\Lambda^\vee$ is a subrepresentation of the local monodromy representation of $\F$. Because the local monodromy representation of $\F$ has at most $m$ isomorphism classes of irreducible subrepresentations, and each of these subrepresentations can only be $\mathcal L_\Lambda^\vee$ for a single $\Lambda$, there are at most $m$ such $\Lambda$. \end{proof}

\begin{lemma}\label{cohweight} The space $H^1_c( \mathbb A^1_{\Fq}, \F \otimes \mathcal L_\Lambda)$ is mixed of weight $\leq w+1$ and is pure of weight $w+1$ for all but at most $m$ values of $\mathcal L_\Lambda$. \end{lemma}

\begin{proof} The sheaf $\F \otimes \mathcal L_\Lambda$ is also a middle extension sheaf and pure of weight $w$, so by \cite[Corollary 3.3.4]{weilii} its cohomology is mixed of weight $\leq w+1.$

Furthermore $H^1_c( \mathbb A^1_{\Fq}, \F \otimes \mathcal L_\Lambda)=H^1(\mathbb P^1_{\Fq}, j_! (\F\otimes \mathcal L_\Lambda))$ so for all $\Lambda$ such that $ j_! (\F\otimes \mathcal L_\Lambda)$ is a middle extension sheaf, its cohomology is pure of weight $w+1$ \cite[Theorem 2]{weilii}. By Lemma \ref{middle}, this holds for all but at most $m$ values of $\Lambda$.
\end{proof}

\begin{lemma}\label{cohomologyvanishing} Assume that $\F$ is not geometrically isomorphic to $\L_{\Lambda^{-1}}$. Then $H^i_c( \mathbb A^1_{\Fq}, \F \otimes \L_{\Lambda} )$ vanishes for $i \neq 1$.

\end{lemma}

\begin{proof} The cohomology group $H^i_c( \mathbb A^1_{\Fq}, \F \otimes \L_{\Lambda} )$ vanishes for all $i$ but $0,1,$ or $2$ because the cohomological dimension of a curve is at most $2$. The cohomology vanishes for $i=0$ because $\F \otimes \L_\Lambda$ is a middle extension sheaf on an affine curve and thus has no compactly supported sections. For $i=2$, $H^i_c( \mathbb A^1_{\Fq}, \F \otimes \L_{\Lambda} )$ is equal to the global monodromy coinvariants of $\F \otimes \L_\Lambda$, which because $\F$ and $\L_\Lambda$ are geometrically irreducible vanish unless $\F $ is geometrically isomorphic to $\L_\Lambda^\vee$, which it cannot be as $\L_{\Lambda}^{\vee} = \L_{\Lambda^{-1}}$. \end{proof}

\begin{lemma}\label{secondmoment} We have \[\lim_{q \to \infty} \frac{ \sum_x |\tr (\Frob_q, \mathcal F,x)|^2}{ q^{w+1}} =1.\] \end{lemma}

\begin{proof} This follows immediately from \cite[Variant Theorem 7.1.2, (1) and (3)]{rls}.\end{proof}

\section{Reduction to the case of middle extension sheaves}

We now prove Theorems \ref{main-arithmetic} and \ref{main-automorphic}, assuming Theorem \ref{main-geometric}.

\begin{proof}[Proof of Theorem \ref{main-arithmetic}]  Let $\rho_1,\dots, \rho_r$ be Galois representations. Because they are finitely ramified, they are each a representation of the \'{e}tale fundamental group of some open subset of $\mathbb A^1_{\mathbb F_{q_0}}$, which defines a lisse sheaf on that open subsets. Let $\F_1,\dots,\F_r$ be these sheaves. The conditions on irreducibility and non-isomorphism statements on $V_1,\dots,V_r$ imply the analogous statements for $\F_1,\dots,\F_r$.

We take $N_1,\dots, N_r$ and $S_q$ as in Theorem \ref{main-geometric}.

Let $\phi_{\Lambda,i}$ be the semisimplified conjugacy class of $\Frob_q /q^{\frac{w_i+1}{2}}$ on $H^1_c( \mathbb A^1_{\overline{\mathbb F}_q}, \F_i \otimes \mathcal L_\Lambda)$. Because semisimplification does not affect the characteristic polynomial, we have $\det ( 1- \phi_{\Lambda,i} q^{ \frac{w_i+1}{2} - s} ) = \det ( 1 - q^{-s} \Frob_q, H^1_c( \mathbb A^1_{\overline{\mathbb F}_q}, \F_i \otimes \mathcal L_\Lambda)) $.  So it suffices to show that $L(\mathbb F_q(x), \rho_i \otimes \Lambda) =  \det ( 1 - q^{-s} \Frob_q, H^i_c( \mathbb A^1_{\overline{\mathbb F}_q}, \F_i \otimes \mathcal L_\Lambda)) $.

This follows from the general formula for the Artin $L$-function of a Galois representation in terms of the Frobenius action on the cohomology of the associated middle extension sheaf, the fact that $\F_i \otimes \mathcal L_\Lambda$ is a middle extension sheaf, and the fact, by construction, that its underlying Galois representation is $\rho_i \otimes \Lambda$. 
\end{proof}

\begin{proof}[Proof of Theorem \ref{main-automorphic}] Let $\pi_1,\dots,\pi_r$ be cuspidal automorphic representations of $GL_{m_i}$ with central character of finite order. Lafforgue associates irreducible Galois representations $\F_1,\dots, \F_r $, with determinant of finite order, to these representations \cite[Theorem VI.9]{Lafforgue}.

Because the determinant has finite order, they are pure of weight zero by \cite[Theorem VII.6(2)]{Lafforgue}.

They are not geometrically isomorphic to each other, and are geometrically irreducible, because of a proof-by-contradiction argument. First note that, because they are irreducible, they are geometrically semisimple, hence to show they are geometrically irreducible, it suffices to show that $\operatorname{Hom}_{\mathbb A^1_{\overline{\mathbb F}_{q_0}}} ( \F_i, \F_i) = \Ql$ as a Galois representation, and to show they are geometrically non-isomorphic, it suffices to show that $\operatorname{Hom}_{\mathbb A^1_{\overline{\mathbb F}_{q_0}}} ( \F_i, \F_j) = 0$ as a Galois representation. Furthermore, because the Frobenius-invariant part of $\operatorname{Hom}_{\mathbb A^1_{\overline{\mathbb F}_{q_0}}} ( \F_i, \F_i) $ is one-dimensional, and the Frobenius action is semisimple, to show it is $\Ql$, it sufficies to show it contains no eigenvector of $\Frob_q$ with eigenvalue $\alpha \neq 1$.  Supposing otherwise, there exists some Frobenius eigenvalue $\alpha$ that appears in the $\operatorname{Hom}_{\mathbb A^1_{\overline{\mathbb F}_{q_0}}} ( \F_i, \F_j) $ , and $\alpha \neq 1$ or $i\neq j$. Then there exists a nontrivial map from $\F_i$ twisted by $\alpha$ to $\F_j$, and becaues they are irreducible this map must be an isomorphism. Because the Langlands correspondence is compatible with tensor products with one-dimensional representations (as follows from the compatibility with local $L$-factors) the corresponding isomorphism on the cusp form side must exist, contradicting the assumption.

We take $N_1,\dots, N_r$ and $S_q$ as in Theorem \ref{main-geometric}.

Let $\phi_{\Lambda,i}$ be the semisimplified conjugacy class of $\Frob_q /q^{\frac{1}{2}}$ on $H^1_c( \mathbb A^1_{\overline{\mathbb F}_q}, \F_i \otimes \mathcal L_\Lambda)$. Because semisimplification does not affect the characteristic polynomial, we have $\det ( 1- \phi_{\Lambda,i} q^{ \frac{1}{2} - s} ) = \det ( 1 - q^{-s} \Frob_q, H^1_c( \mathbb A^1_{\overline{\mathbb F}_q}, \F_i \otimes \mathcal L_\Lambda))  = \det ( 1 - q^{-s} \Frob_q, H^1_c( \mathbb P^1_{\overline{\mathbb F}_q}, j_* ( \F_i \otimes \mathcal L_\Lambda) ))  = L ( \pi_i \otimes \Lambda, q^{- s})  $ by \cite[Theorem VI.9(2)]{Lafforgue}. So the equidistribution follows from Theorem \ref{main-geometric}, once we note that Lafforgue's definition of the automorphic $L$-function differs from the classical defintion by the substitution $Z = q^{-s}$. \end{proof}
\section{Sheaf construction}\label{sheaf-construction}

This section is the analogue of \cite[Section 4]{wvqkr}.

Let $r$ be the highest natural number with $p^r \leq d$. View $\mathbb A^d$ as the space parameterizing tuples of a polynomial $f_1$ in $x$ of degree at most $d$, a polynomial $f_2$ in $x$ of degree at most $d/p$, $\dots$, up to a polynomial $f_r$ of degree at most $d/p^4$, such that the coefficients of $f_1 ,\dots, f_r$ of $x^n$ always vanish if $n$ is a multiple of $p$. In other words, in coordinates $a_1,\dots,a_d$, we have $f_i = \sum_{k |  k \neq 0 \mod p, k p^r \leq n}  x^k a_{p^i k}$.  Take a character $\psi: \mathbb Z/p^{r+1} \to \Ql^\times$. Form the $\mathbb Z/p^{r+1}$ covering of $\mathbb A^1 \times \mathbb A^d$ defined by the Witt vector equation $z- F(z) = (f_1,f_2,\dots,f_r)$, with $F$ the Witt vector Frobenius. Its push out by $\psi$ gives us the Artin-Schreier-Witt sheaf $\L_{\psi,r,x} ( f_1(x),f_2(x),\dots,f_r(x))$, a lisse rank $1$ sheaf on $\mathbb A^1 \times \mathbb A^d$.

 Let $\L_{univ} = \L_{\psi,r,x} ( f_1(x),f_2(x),\dots,f_r(x))$ be this sheaf on $\mathbb A^1 \times \mathbb A^d$.

\begin{lemma}\label{universal} For each field extension $\mathbb F_q$ of $\mathbb F_{q_0}$, there is a unique bijection between points $y \in \mathbb A^d(\mathbb F_q)$ and characters $\Lambda: W_{d,q}\to \Ql^\times$ that sends each point $y$ to a character $\Lambda$ such that the restriction $\L_{univ,y}$ of $\L_{univ}$ to the fiber over $y$ is isomorphic to $\L_{\Lambda}$.\end{lemma}

\begin{proof}  For each point $y$ in $\mathbb A^d(\mathbb F_q)$, $\L_{univ,y}$ is a rank one lisse sheaf on $\mathbb A^1$ i.e. a character $\pi_1(\mathbb A^1_{\mathbb F_q}) \to \Ql^{\times}$. By \cite[Theorem 1]{brylinski}, this character has conductor exponent at most $d+1$ at $\infty$. Because the polynomials $f_1,\dots, f_r$ have zero constant term, they vanish at $x=0$, so the action of Frobenius on the stalk of  $\mathcal L_{univ,y}$ at the point $0$ is trivial.  Hence the monodromy character of $\L_{univ,y}$ factors through the Galois group of the maximal abelian extension of $\mathbb F_q(t)$, unramified away from $\infty$, split at $0$, and with conductor exponent at most $d+1$ at $\infty$. By Lemma \ref{cft} and \ref{cft-galois}, this character factors uniquely through $W_{d,q}$.

It remains to check the existence of a bijection. In other words, it is sufficient to check that each all characters of $W_{d,q}$ appear exactly once this way. Because $W_{d,q}$ is a quotient of the inertia group $\mathbb F_q[[t]]^\times$ at $\infty$ in the Galois group of the maximal ableian extension, we can work with characters of $\mathbb F_q[[t]]^\times$. By \cite[Theorem 1]{brylinski}, two different Witt vectors $(f_1(x),\dots,f_r(x))$ and $(f_1'(x),\dots,f_r'(x))$ produce the same character only if we have \[(f_1(x),\dots,f_r(x))-(f_1'(x),\dots,f_r'(x))= ( a_1(x)^p,\dots, a_r(x)^p ) - (a_1(x),\dots, a_r(x)) \]  in the ring of Witt vectors for some $(a_1,\dots,a_r)$. That can only happen if the highest degree that $f$ and $f'$ differ in is equal to the highest degree of $(a_1(x)^p,\dots,a_r(x)^p)$ and thus is multiple of $p$, which is impossible for $( f_1(x),f_2(x),\dots,f_r(x))$ as they are $0$ in every degree a multiple of $p$. Thus all characters of conductor at most $d$ appear at most once this way. Then because there are $q^d$ points of $\mathbb A^d(\mathbb F_q)$, and $q^d$ characters $W_{d,q}\to \Ql^\times$, all characters must appear exactly once this way, as desired.

 \end{proof}
 
 While we have constructed this bijection by constructed the character $\Lambda$ from the point $y$, in applying it will be more convenient to use the reverse, and we therefore set $y_{\Lambda}$ as the point of $\mathbb A^d(\mathbb F_q)$ corresponding to $\Lambda$.

Let $\F$ be a geometrically irreducible middle extension sheaf on $\mathbb A^1_{\mathbb F_{q_0}}$, pure of weight $w$.

\begin{defi} Let $pr_1: \mathbb A^1 \times \mathbb A^d \to \mathbb A^1$ and $pr_2: \mathbb A^1 \times \mathbb A^d \to \mathbb A^d$ be the projections. Let  \[ \mathcal G=  R^1pr_{2!} (pr_1^* \F \otimes \L_{univ}).\] \end{defi}

\begin{lemma}\label{shconstruction}For each field extension $\mathbb F_q$ of $\mathbb F_{q_0}$ and each character $\Lambda: W_{d,q} \to \Ql^\times$, \[\mathcal G_{y_\lambda} = H^1_c(\mathbb A^1_{\Fq}, \F \otimes \L_\Lambda).\]  \end{lemma}

\begin{proof} By the proper base change theorem and Lemma \ref{universal}, \[\mathcal G_y = H^1_c(\mathbb A^1_{\Fq}, \F \otimes \L_{univ,y})=  H^1_c(\mathbb A^1_{\Fq}, \F \otimes \L_\Lambda)\] \end{proof}

\begin{lemma}\label{swan-lemma} Let $V_1, \dots, V_t$ be the irreducible Jordan-H\"{o}lder factors of the inertia representation of $\mathcal F$ at $\infty$, and let $s_1,\dots,s_t$ be their slopes.

There exists a  unique subset $T \subset \overline{\mathbb F}_q$, of size at most $t$ (and therefore at most $m$), such that for $y = (a_1,\dots,a_{d})\in \mathbb A^d$, if $a_d \not\in T$, the Swan conductor at $\infty$ of $\F \otimes \L_{univ,y}$ is equal to $\sum_{i=1}^t \max ( d, s_i) \dim V_i $, and if $a_d \in T$, the Swan conductor is less than this.

Moreover, if $d$ is less than the lowest slope of the $V_i$s, $T$ is empty, and if $d$ is greater than the highest slope of the $V_i$s, we $T =\{0\}$. \end{lemma}

\begin{proof} $T$ is unique because we can determine whether $a \in T$ for $T$ satisfying this condition by calculating the Swan conductor of $\F \otimes \L_{univ,y}$ at a point $y$ with $a_d=a$.

Let $V_i$ be one such an irreducible component. Then $V_i \otimes \L_{univ,y}$ is irreducible of slope at most $\max(d,s_i)$. From this it follows that the Swan conductor of $\F \otimes \L_{univ,y}$ is at most $\sum_{i=1}^t \max ( d, s_i) \dim V_i $, and it is equal to this sum if and only if the slope of $ V_i \otimes \L_{univ,y}) $ is $\max(d,s_i)$ for all $i$.

 If $s_i>d$ then the slope of $ V_i \otimes \L_{univ,y}$ is automatically  $ s_i \max(d,s_i)$. Assume that for some $y$, the slope of $V_i \otimes \L_{univ,y}$ is less than $\max(s_i,d)$. It follows that $\max(s_i,d)=d$. Then we will show that for any $y' \in \mathbb A_d$,  the slope of  $ V_i \otimes \L_{univ,y'} $ is less than $d$ if and only if $a_d(y)=a_d(y')$. 
 
 To do this, note that $\mathcal L_{univ,y'} \otimes \mathcal L_{univ,y}^{-1} $ can be expressed as an Artin-Schreier-Witt sheaf of a Witt vector obtained by subtracting the Witt vectors of $y$ from $y'$. This is a Witt vector of degree at most $d$, whose leading term is $a_d(y')-a_d(y)$. If this leading term vanishes, then the slope of  $\mathcal L_{univ,y'} \otimes \mathcal L_{univ,y}^{-1} $  is less than $d$, and so the slope of \[ V_i \otimes \L_{univ,y'} =\left( \mathcal L_{univ,y'} \otimes \mathcal L_{univ,y}^{-1}  \right) \otimes  \left( V_i \otimes \L_{univ,y}\right)\] is less than $d$. If this leading term is nonvanishing, then the slope of  $\mathcal L_{univ,y'} \otimes \mathcal L_{univ,y}^{-1} $  is $d$, so because the slope of $\left( V_i \otimes \L_{univ,y}\right)$ is less than $d$, their tensor product has slope $d$, as desired.
 
 Thus let $T$ be the union over $i$ such that $s_i=d$ of the $a_d$ value of some point where the slope of $ V_i \otimes \L_{univ,y}$ is less than their maximum value, if it exists. By our calculation of the local slopes, the upper bound is attained if and only if $a_d$ is not equal to any of the elements of $T$.  Furthermore, $T$ is empty if $d< s_i$ for all $i$.
 
 It remains to check that if $d>s_i$ for all $i$ then $T$ consists of the $0$. It suffices to check that if $a_d(y)=0$, then $V_i \otimes \L_{univ,y}$ has slope $< \max(d,s_i)=d$. This follows because $\L_{univ,y}$ has slope less than $d$ in this case, and we assumed $V_i$ has slope less than $d$.
 
 \end{proof}

 \begin{defi} Let $U$ be the subset of $\mathbb A^d_{\mathbb F_q}$ consisting of points where $a_d \not\in T$.\end{defi}

 \begin{lemma}\label{swan-lisse} The sheaf $\mathcal G$ is lisse on $U$. \end{lemma}
 
 \begin{proof} Because $\mathcal L_{univ}$ is lisse on $\mathbb A^1$,  $\mathcal F \otimes \L_{univ}$ is lisse away from a union of horizontal divisors, and has constant swan conductors on those divisors. By Lemma \ref{swan-lemma}, the Swan conductor of $\mathcal F \otimes \L_{univ}$ at $\infty$ is constant on $U$.
 Thus by Deligne's semicontinuity theorem \cite[Theorem 2.1.1]{LaumonSMF}, $\mathcal G$ is lisse on $U$. \end{proof}

\begin{lemma}\label{shweight} The sheaf $\mathcal G$ is everywhere mixed of weight $\leq w+1$, and its restriction to $U$ is pure of weight $w+1$. \end{lemma}

\begin{proof} The sheaf $\F$ is mixed of weight $\leq w$ and $\L_{univ}$ is pure of weight $0$ (because its arithmetic monodromy factors through $\mathbb Z/p^{r+1}$ and so is finite) , so $pr_1^* \F \otimes \L_{univ}$ is mixed of weight $\leq w$, so $\mathcal G$ is mixed of weight $\leq w+1$ by \cite[Theorem 1]{weilii}.

Let $\mathbb F_q$ be an extension of $\mathbb F_{q_0}$ large enough that $U(\mathbb F_q)$ has greater than $m$ points. Then on at least one point of $U(\mathbb F_q)$, the stalk of $\mathcal G$ is pure of weight $w+1$ by Lemmas \ref{cohweight} and \ref{shconstruction}. Hence by \cite[Corollary 1.8.12]{weilii}, $\G$ is pure of weight $w+1$ on $U$. \end{proof}

\begin{defi} Let $N$ be the rank of $\mathcal G$ on $U$.\end{defi}

 We will shortly give a concrete lower bound for $N$, and an exact formula as long as $d$ is sufficiently large with respect to $\F$, which should be useful in applications.

\begin{defi} For each point $x \in \mathbb A^1_{\Fq}$, let $\operatorname{drop}(\F,x)$ be the drop of $\F$ at $x$ and let $\operatorname{swan}(\F,x)$ be the Swan conductor of $\F$ at $x$. Then let \[c_\F = \sum_{x \in \mathbb A^1(\Fq)} (\operatorname{drop}(\F,x)+ \operatorname{swan}(\F,x)).\] In other words, $c_\F$ is the logarithmic Artin conductor of $\F$, not counting the contribution at $\infty$.\end{defi}

\begin{lemma}\label{rank formula} We have \[N \geq m(d-1) + c_{\F}\] and if $d$ is at least the highest slope of $\F$ at $\infty$, then this is an equality.

\end{lemma}

\begin{proof} In view of Lemma \ref{cohomologyvanishing}, $N = - \chi(\mathbb A^1_{\Fq}, \F \otimes \L_\Lambda)$ for a character $\Lambda$ corresponding to a point in $U(\mathbb F_q)$ for some $q$. The Euler characteristic formula for a middle extension sheaf on a curve then shows that

\[ -\chi(\mathbb A^1_{\Fq}, \F \otimes \L_\Lambda)=  \sum_{x \in \mathbb A^1(\Fq)} (\operatorname{drop}(\F \otimes \L_{\lambda},x)+ \operatorname{swan}(\F\otimes \L_\Lambda,x)  ) + \operatorname{swan}(\F \otimes \L_\Lambda,\infty) -  \chi(\mathbb A^1) \operatorname{rank}(\F \otimes \L_\Lambda)\]

Because $\L_\Lambda$ is lisse of rank $1$ on $\mathbb A^1$, $\operatorname{drop}(\F \otimes \L_{\lambda},x)= \operatorname{drop}(\F ,x)$ and $\operatorname{swan}(\F \otimes \L_{\lambda},x)= \operatorname{swan}(\F ,x)$. 

Similarly $\operatorname{rank}(\F \otimes \L_\Lambda)= \operatorname{rank}(\F)=m$ and $\chi(\mathbb A^1)$ is $1$, so the final term contributes $-m$.

Then by Lemma \ref{swan-lemma}, as long as $y_{\Lambda}$ is in $U$ of $\mathbb A^d$, $\operatorname{swan}( \F \otimes \L_\Lambda,\infty) \geq m d$, and $=md$ if all slopes of $\F$ at $\infty$ are less than $d$. 
\end{proof}

\begin{example} We provide an example of how to compute the rank $N$ in a special case.  Let $E$ be an elliptic curve over $\mathbb F_{q_0}(x)$, so that its Tate module is a rank $2$ Galois representation, and let $\mathcal F$ be the associated middle extension sheaf. Assume that the characteristic of $\mathbb F_{q_0}$ is greater than $3$. Because of this simplyifiying assumption, the Tate module of $E$ is tamely ramified. In particular, its highest slope at $\infty$ is $0$, so Lemma \ref{rank formula} implies that $N = 2 (d-1) + c_\mathcal F$ for all $d$.

In this case, $c_\mathcal F$ is the sum over all geometric points of the drop of the rank of $\mathcal F$ at that point, which is one if $E$ has semistable reduction there and two if $E$ has additive reduction there. In other words, it is the total degree of all the primes of $\mathbb F_{q_0} [x]$ at which $E$ has semistable reduction, plus twice the total degree of all the primes of $\mathbb F_{q_0}[x]$ at which $E$ has additive reduction. This matches the usual formula for the conductor of an elliptic curve, except we add the degrees of the primes instead of multiplying their norms. Here we do not count the prime at $\infty$.

For instance, in the special case of the Legendre curve $y^2 = z(z-1)(z-x)$, whose only primes of bad reduction are $(x)$ and $(x-1)$, both stable, $c_\mathcal F=2$. \end{example}

\begin{lemma} Let $\pi$ be an automorphic representation of $GL_m(\mathbb F_{q_0}(x))$ with central character of finite order and let $\mathcal F$ be the middle extension sheaf associated to the Galois representation corresponding to $\pi$ under the Langlands correspondence.   Let $g(x)$ be a polynomial such that $\pi$ is new of level $g$, in the sense that for all prime polynomials $\mathfrak p$ of $\mathbb F_{q_0}(x)$, $\pi_\mathfrak p$ contains a vector invariant under the mirabolic subgroup mod $\mathfrak p^k$, $k$ the highest power of $\mathfrak p$ dividing $g$, but not modulo any smaller power.  Then \[ N \geq m(d-1) + \deg g\] with equality if $d$ is greater than the depth of $\pi_\infty$. \end{lemma}

\begin{proof} In view of Lemma \ref{rank formula}, it suffices to show that $c_{\mathcal F}= \deg g$ and that the highest slope of $\mathcal F$ at $\infty$ is the depth of $\pi_\infty$. The first assertion follows from \cite[Theorem VII.3(i)]{Lafforgue} because the logarithmic level or conductor is the degree of the local epsilon factor.  The second assertion is \cite[Proposition 4.2]{llSLN}.\end{proof}

\section{Monodromy calculation}

We now calculate the monodromy of the sheaf $\mathcal G$, as well as the combine monodromy of multiple such sheaves. This is the analogue of \cite[Section 6]{wvqkr}.

\begin{lemma}\label{mmp}
Let $k$ be a natural number. Assume that $k \leq d$ and, if $\F$ is geometrically isomorphic to $\L_{\Lambda_0}$ for some $\Lambda_0: W_{d,q} \to \Ql^\times$ then $k<d$.

\[\lim_{q \to \infty} \frac{ 1}{ q^d}  \sum_{ \Lambda: W_{d,q}\to \Ql^\times} \frac{ \left| \tr (\Frob_q, H^1_c( \mathbb A^1_{\Fq}, \F \otimes \L_\Lambda)) \right|^{2k}}{q^{k(w+1)}} =   k!.\]

\end{lemma}

\begin{proof}  First we show that

\[\lim_{q \to \infty} \frac{ 1}{ q^d}  \sum_{ \Lambda: W_{d,q}\to \Ql^\times} \frac{ \left|  \sum_{x \in \mathbb F_q} \tr(\Frob_q, \F, x) \Lambda(1- tx) \right|^{2k}}{q^{k(w+1)}} =   k!.\]

This follows from Lemma \ref{diophantinepart} once we check the conditions. The first condition, that $\tr(\Frob_q, \F, x) = O(q^{w/2})$, is satisfied as $\F$, being pure of weight $w$, is pointwise mixed of weight $\leq w$, i.e. all eigenvalues of Frobenius on the stalk at a point are at most $q^{w/2}$ and thus the trace at a point is at most $m q^{w/2}$. The second condition is Lemma \ref{secondmoment}. 

For each $\Lambda$, if $\F$ is not geometrically isomorphic to $\L_{\Lambda^{-1}}$, then we have by the Lefschetz formula and Lemma \ref{cohomologyvanishing}

\[ \sum_{x \in \mathbb F_q} \tr(\Frob_q, \F, x) \Lambda(1- tx) = \sum_{x \in \mathbb F_q} \tr(\Frob_q, \F \otimes \L_\Lambda, x) \]\[= \sum_{i=0}^2 (-1)^i \tr(\Frob_q, H^i_c(\mathbb A^1_{\Fq}, \F \otimes \L_\Lambda))= -\tr(\Frob_q, H^i_c(\mathbb A^1_{\Fq}, \F \otimes \L_\Lambda)) \]  

In particular, if $\F$ is not geometrically isomorphic to $\L_{\Lambda_0}$ for any $\Lambda_0 : W_{d,q} \to \Ql^\times$, then we are done as $\left| \sum_{x \in \mathbb F_q} \tr(\Frob_q, \F, x) \Lambda(1- tx)\right| = \left| \tr(\Frob_q, H^i_c(\mathbb A^1_{\Fq}, \F \otimes \L_\Lambda)) \right|$ for all $\Lambda$.

If $\F$ is geometrically isomorphic to $\L_{\Lambda_0}$, we need an additional argument. The action of $\Frob_q$ on $\F$ at $x$ is equal to $\Lambda(1-tx)$ times a constant $\alpha$ satisfying $|\alpha| =q^{w/2}$. Then  $\left| \sum_{x \in \mathbb F_q} \tr(\Frob_q, \F, x) \Lambda(1- tx)\right| = \left| \tr(\Frob_q, H^i_c(\mathbb A^1_{\Fq}, \F \otimes \L_\Lambda)) \right|$ for all $\Lambda$ except $\Lambda=\Lambda_0^{-1}$. In this case, $\sum_{x \in \mathbb F_q} \tr(\Frob_q, \F, x) \Lambda_0^{-1} (1- tx) = \sum_{x \in \mathbb F_q} \alpha = \alpha q$, but $H^1_c( \mathbb A^1_{\Fq}, \F \otimes \L_\Lambda)= H^1_c(\mathbb A^1_{\Fq}, \Ql)=0$. Hence

\[  \sum_{ \Lambda: W_{d,q}\to \Ql^\times} \frac{ \left| \tr (\Frob_q, H^1_c( \mathbb A^1_{\Fq}, \F \otimes \L_\Lambda)) \right|^{2k}}{q^{k(w+1)}} \]

\[= \sum_{ \Lambda: W_{d,q} \to \Ql^\times} \frac{ \left| \sum_{x \in \mathbb F_q} \tr(\Frob_q, \F, x) \Lambda(1- tx) \right|^{2k}}{q^{k(w+1)}} - \frac{ \left| \alpha q\right|^{2k}}{q^{k(w+1)}}.\]

We know the first term, divided by $q^d$, converges to $k!$ as $q$ goes to $\infty$. It remains to evaluate the second term. We have 
\[\frac{ \left| \alpha q\right|^{2k}}{q^{k(w+1)}}= \frac{ q^{2k(1+w/2)}}{q^{k(w+1)}}= q^k\]
and hence, divided by $q^d$, it converges to $0$ as long as $k<d$.

\end{proof}

Let $G_{geom} \subseteq GL_N$ be the geometric monodromy group of $\mathcal G$ on $U$ with its $N$-dimensional standard representation $V$. Because $\G$ is pure of weight $w+1$ on $U$, it is geometrically semisimple on $U$ \cite[Corollary 3.4.12]{weilii}, so $G_{geom}$ is reductive.

\begin{prop}\label{moments} Let $k$ be a natural number. Assume that $k \leq d$ and, if $\F$ is geometrically isomorphic to $\L_{\Lambda_0}$ for some $\Lambda_0: W_{d,q} \to \Ql^\times$ then $k<d$.
Then the dimension of the $G_{geom}$-invariants of $V^{\otimes k} \otimes V^{\vee \otimes k}$  is $k!$. \end{prop}

\begin{proof} This invariant subspace is
\[H^{2d}_c \left(U_{\overline{\mathbb F}_q}, \mathcal G^{\otimes k} \otimes \left(\mathcal G^{\vee}\right)^{\otimes k} \right)\]
and is pure of weight $2d$, because $\mathcal G \otimes \mathcal G^{\vee}$ is pure of weight $0$. So the dimension of this invariant subspace is equal to 

\[ \lim \sup_ {q \to \infty} \frac{1}{q^{d}} \left| \tr\left(\Frob_q, H^{2d}_c\left(U_{\overline{\mathbb F}_q}, \mathcal G^{\otimes k} \otimes \left( \mathcal G^{\vee}\right)^{\otimes k} \right)\right)\right| \]

\[= \lim \sup_ {q \to \infty} \frac{1}{q^{d}} \left|\sum_i (-1)^i\tr\left(\Frob_q, H^{i}_c\left(U_{\overline{\mathbb F}_q}, \mathcal G^{\otimes k} \otimes \left( \mathcal G^{\vee}\right)^{\otimes k} \right)\right)\right| \]

\[= \lim \sup_ {q \to \infty} \frac{1}{q^{d}} \left|\sum_{x \in U(\mathbb F_q)} \tr\left(\Frob_q, \mathcal G^{\otimes k}_x \otimes \left( \mathcal G^{\vee}_x\right)^{\otimes k} \right)\right| \]

\[= \lim \sup_ {q \to \infty} \frac{1}{q^{d}} \left|\sum_{x \in U(\mathbb F_q)} \tr(\Frob_q, \mathcal G_x)^k \tr(\Frob_q,\mathcal G^\vee_x)^k \right|\]

(Because $\mathcal G$ is pure of weight $w+1$ on $U$ (Lemma \ref{shweight}),  $\tr(\Frob_q, \mathcal G^\vee)= \overline{\tr(\Frob_q, \mathcal G)}/q^{w+1}$)

\[= \lim \sup_ {q \to \infty} \frac{1}{q^{d+k(w+1)}} \left|\sum_{x \in U(\mathbb F_q)} \left|\tr(\Frob_q,  \mathcal G_x )\right|^{2k}\right| \]

(If we extend the sum from $U$ to $\mathbb A^d$, by Lang-Weil or even simpler estimates we add $O(q^{d-1})$ terms. Because $\G$ is mixed of weight $\leq w+1$ (Lemma \ref{shweight}), each of those terms has size at most $(O (q^{(w+1)/2}))^{2k}  = O(q^{k(w+1)})$, so the total contribution $\frac{1}{q^{d+k(w+1)}}  O( q^{ d-1+ k(w+1)})$ is $o(1)$ and does not affect the limit.)

\[= \lim \sup_ {q \to \infty} \frac{1}{q^{d+k(w+1)}} \left|\sum_{x \in \mathbb A^d(\mathbb F_q)} \left|\tr(\Frob_q,  \mathcal G_x )\right|^{2k}\right| \]

\[=\lim_{q \to \infty} \frac{ 1}{ q^d}  \sum_{ \Lambda: W_{d,q}\to \Ql^\times} \frac{ \left| \tr (\Frob_q, H^1_c( \mathbb A^1_{\Fq}, \F \otimes \L_\Lambda)) \right|^{2k}}{q^{k(w+1)}} =   k!\] by Lemma \ref{mmp}.\end{proof}

\begin{cor}\label{monodromy} If $d\geq 4$ then $G_{geom}$ contains $SL_N$. \end{cor}

\begin{proof} First assume that $\mathcal F$ is not isomorphic to $\mathcal L_{\Lambda_0}$ for some $\Lambda_0: W_{d,q} \to \Ql^\times$ 

By Proposition \ref{moments}, the dimension of the $G_{geom}$ invariants of  $V^{\otimes 4} \otimes V^{\vee \otimes 4}$ is $4!$. Guralnick and Tiep's solution of Larsen's conjecture \cite[Theorem 1.4]{guralnicktiep}, which states that this condition, together with the assumption $N\geq 5$, implies that $G_{geom}$ contains $SL_N$. However, the same thing is proved in the case $N=3$ and $N=3$ in \cite[Theorem 2.12]{guralnicktiep}, noting that in characteristic zero the only exceptional case that applies is case (ii), which requires $N=2$. By Lemma \ref{rank formula} $N \geq (d-1) \geq 3$ so this assumption is satisfied.

In the case $\mathcal F$ is isomorphic to $\mathcal L_{\Lambda_0}$, twisting by $\mathcal F \otimes \mathcal L_{\Lambda} = \mathcal L_{\Lambda \Lambda_0}$ and $\Lambda \mapsto \Lambda\Lambda_0$ gives an automorphism of the space $\mathbb A^d$ of characters. Pulling back Katz's sheaf $L_{univ}$ under that automorphism gives a sheaf isomorphic to $\mathcal G$, so the monodromy groups are the same, but the monodromy group of $\mathcal G$ was already proven to contain $SL_N$ by \cite[Theorem 5.1]{wvqkr}
  \end{proof}

Let $\F_1$ and $\F_2$ be two geometrically irreducible middle extension sheaves on $\mathbb A^1_{\mathbb F_q}$, pure of weights $w_1$ and $w_2$ respectively. Assume that $\F_1$ and $\F_2$ are not geometrically isomorphic. 

\begin{lemma}\label{mmtwo} If $d \geq 2$ then  
\[\lim_{q \to \infty} \frac{ 1}{ q^d}  \sum_{ \Lambda: W_{d,q}\to \Ql^\times} \frac{ \left| \tr (\Frob_q, H^1_c( \mathbb A^1_{\Fq}, \F_1 \otimes \L_\Lambda)) \right|^2  \left| \tr (\Frob_q, H^1_c( \mathbb A^1_{\Fq}, \F_2 \otimes \L_\Lambda)) \right|^2  }{q^{w_1 +w_2+2}} =   1.\]

\end{lemma}

\begin{proof} We split into cases depending on whether $\F_1$ or $\F_2$ is geometrically isomorphic to $\L_\Lambda$ for any $\Lambda: W_{d,q}\to \Ql^\times$.

In the case that neither $\F_1$ nor $\F_2$ is geometrically isomorphic to such a sheaf, by the Lefschetz formula and Lemma \ref{cohomologyvanishing}

\[ \sum_{x \in \mathbb F_q} \tr(\Frob_q, \F_i, x) \Lambda(1- tx) = \sum_{x \in \mathbb F_q} \tr(\Frob_q, \F_i \otimes \L_\Lambda, x) \]\[= \sum_{i=0}^2 (-1)^i \tr(\Frob_q, H^i_c(\mathbb A^1_{\Fq}, \F_i \otimes \L_\Lambda))= -\tr(\Frob_q, H^i_c(\mathbb A^1_{\Fq}, \F_i \otimes \L_\Lambda)) \]  so this follows from Lemma \ref{indio} once we check the conditions. The first condition, that $\tr(\Frob_q, \F_i, x) = O(q^{w/2})$, is satisfied as $\F$, being pure of weight $w$, is punctually mixed of weight $\leq w$, i.e. all eigenvalues of Frobenius on the stalk at a point are at most $q^{w/2}$ and thus the trace at a point is at most $m q^{w/2}$. The second condition is Lemma \ref{secondmoment}. The third condition follows from \cite[Theorem 1.7.2(2)]{mmp} as after twisting by $q^{(w_1+1)/2}$ and $q^{(w_2+1)/2}$ respectively, $\F_1[1]$ and $\F_2[1]$ are perverse and pure of weight $0$ with no common constituents.

In the case where at least one of the $\F_i$ is geometrically isomorphic to $\L_{\Lambda_i}$, the equation \[ \sum_{x \in \mathbb F_q} \tr(\Frob_q, \F_i, x) \Lambda(1- tx) = -\tr(\Frob_q, H^i_c(\mathbb A^1_{\Fq}, \F_i \otimes \L_\Lambda)) \] holds for all but one or two values of $\Lambda$. So the same result will be correct as long as the contribution from these exceptional values converges to $0$ as $q$ goes to $\infty$.  These values occur when $\F_i$ is geometrically isomorphic to $\L_{\Lambda}^{-1}$, so the right side of this equation vanishes. Hence the contribution of each exceptional term is 

\[ \frac{1}{q^d} { \left| \sum_{x \in \mathbb F_q} \tr(\Frob_q, \F_1,x) \Lambda(1-tx) \right|^2  \left| \sum_{x \in \mathbb F_q}  \tr (\Frob_q, \tr(\Frob_q, \F_2,x) \Lambda(1-tx)  \right|^2  }{q^{w_1 +w_2+2}} \]

Assume without loss of generality that $\F_1$ is geometrically isomorphic to $\L_\Lambda^{-1}$. Then as $\F_1$ is not geometrically isomorphic to $\F_2$,  $\F_2$ is not geometrically isomorphic to $\L_\Lambda$, \[ \left| \sum_{x \in \mathbb F_q}  \tr (\Frob_q, \tr(\Frob_q, \F_2,x) \Lambda(1-tx)  \right|^2 = O(q^{w_2+1})\] by Lemma \ref{cohweight} and Lemma \ref{cohomologyvanishing}. Then as $\tr(\Frob_q,\F_q,x) = O(q^{w_1/2})$, \[ \left| \sum_{x \in \mathbb F_q} \tr(\Frob_q, \F_1,x) \Lambda(1-tx) \right|^2 =O(q^{w_1+2})\] and so the contribution is $O( q^{1-d})$. As $d\geq 2$, this converges to $0$, as desired.
\end{proof}

Again let $\F_1$ and $\F_2$ be two middle extension sheaves on $\mathbb A^1_{\mathbb F_q}$, pure of weights $w_1$ and $w_2$ respectively. Form the corresponding sheaves $\G_1$ and $\G_2$, as in Lemma \ref{shweight}. Let $U$ be an open subset on which both sheaves are lisse. Let $N_1$ and $N_2$ be the ranks of $\G_1$ and $\G_2$, and let $G_{geom}$ be the geometric monodromy group of $\G_1+\G_2$, acting on standard representations $V_1$ and $V_2$ corresponding to $\G_1$ and $\G_2$.

\begin{lemma}\label{indmoments}  Assume that neither $\F_1$ nor $\F_2$ is isomorphic to $\L_{\Lambda}$ for any character $\Lambda:W_{d,q}\to \Ql^\times$.  Assume that $d \geq 2$.

The dimension of the $G_{geom}$-invariants of $V_1 \otimes V_1^{\vee} \otimes V_2 \otimes V_2^\vee$  is $1$. \end{lemma}

\begin{proof} This invariant subspace is
\[H^{2d}_c \left(U_{\overline{\mathbb F}_q}, \G_1 \otimes \G_1^\vee \otimes \G_2 \otimes \G_2^\vee\right)\]
and is pure of weight $2d$, because $\mathcal G_i \otimes \mathcal G_i^{\vee}$ is pure of weight $0$. So the dimension of this invariant subspace is equal to 

\[ \lim \sup_ {q \to \infty} \frac{1}{q^{d}} \left| \tr\left(\Frob_q, H^{2d}_c\left(U_{\overline{\mathbb F}_q}, \G_1 \otimes \G_1^\vee \otimes \G_2 \otimes \G_2^\vee \right)\right)\right| \]

\[= \lim \sup_ {q \to \infty} \frac{1}{q^{d}} \left|\sum_i (-1)^i\tr\left(\Frob_q, H^{i}_c\left(U_{\overline{\mathbb F}_q}, \G_1 \otimes \G_1^\vee \otimes \G_2 \otimes \G_2^\vee \right)\right)\right| \]

\[= \lim \sup_ {q \to \infty} \frac{1}{q^{d}} \left|\sum_{x \in  U(\mathbb F_q)} \tr\left(\Frob_q, \G_1 \otimes \G_1^\vee \otimes \G_2 \otimes \G_2^\vee \right)\right| \]

\[= \lim \sup_ {q \to \infty} \frac{1}{q^{d}} \left|\sum_{x \in U(\mathbb F_q)} \tr(\Frob_q, \G_1)\tr(\Frob_q, \G_1^\vee)\tr(\Frob_q, \G_2)\tr(\Frob_q, \G_2^\vee) \right|\]

(Because $\mathcal G_i$ is pure of weight $w_i+1$ on $U$ (Lemma \ref{shweight}),  $\tr(\Frob_q, \mathcal G_i^\vee)= \overline{\tr(\Frob_q, \mathcal G_i)}/q^{w_i+1}$)

\[= \lim \sup_ {q \to \infty} \frac{1}{q^{d+w_1+w_2+2}} \left|\sum_{x \in U(\mathbb F_q)} \left|\tr(\Frob_q,  \G_1 )\right|^{2}  \left|\tr(\Frob_q,  \G_2 )\right|^{2} \right| \]

(If we extend the sum from $U$ to $\mathbb A^d$, by Lang-Weil or even simpler estimates we add $O(q^{d-1})$ terms. Because $\G_i$ is mixed of weight $w_i+1$ (Lemma \ref{shweight}), each of those terms has size at most $O( q^{w_1+w_2+2})$, so the total contribution $\frac{1}{q^{d+w_1+w_2+2}}  O( q^{ d-1+ w_1+w_2+2})$ is $o(1)$ and does not affect the limit.)

\[= \lim \sup_ {q \to \infty} \frac{1}{q^{d+w_1+w_2+2} }\left|\sum_{x \in \mathbb A^d(\mathbb F_q)} \left|\tr(\Frob_q,  \mathcal G_1 )\right|^2  \left|\tr(\Frob_q,  \mathcal G_2 )\right|^2   \right| \]

\[=\lim_{q \to \infty} \frac{ 1}{ q^d}  \sum_{ \Lambda: W_{d,q}\to \Ql^\times} \frac{ \left| \tr (\Frob_q, H^1_c( \mathbb A^1_{\Fq}, \F_1 \otimes \L_\Lambda)) \right|^{2} \left| \tr (\Frob_q, H^1_c( \mathbb A^1_{\Fq}, \F_2 \otimes \L_\Lambda)) \right|^{2}}{q^{w_1+w_2+2}} =  1\] (Lemma \ref{mmtwo})\end{proof}

Let $\F_1,\dots, \F_r$ be middle extension sheaves on $\mathbb A^1_{\mathbb F_q}$, pure of weights $w_1, \dots, w_r$ respectively. Form the corresponding sheaves $\G_1,\dots, \G_r$, as in Lemma \ref{shweight}. Let $U$ be an open subset on which all sheaves are lisse. Let $N_1,\dots,N_r$ be the ranks of $\G_1,\dots,\G_r$, and let $G_{geom}$ be the geometric monodromy group of $\bigoplus_{i=1}^r \G_i$.

\begin{cor} Assume that $d \geq 4$. Then the composition \[ G_{geom} \to \prod_{i=1}^r GL_{N_i} \to \prod_{i=1}^r PGL_{N_i} \] is a surjection.

\end{cor}

\begin{proof}Corollary \ref{monodromy} implies that, for all $i$, the image of $G_{geom}$ inside $PGL_{N_i}$ contains the image of $SL_{N_i}$ inside $PGL_{N_i}$, which is all of $PGL_{N_i}$. Hence the projection map from the image of $G_{geom}$ inside $PGL_{N_i}$ to each factor is a surjection. We apply Goursat-Kolchin-Ribet \cite[Proposition 1.8.2]{esde} to the image of $G_{geom}$ inside $\prod_{i=1}^m PGL_{N_i} $. Because each $PGL_{N_i} $ is a simple group, it follows that, to show that the image is the whole product, it suffices to show that there is no $i,j$ with $i\neq j$ and isomorphism $\sigma: PGL_{N_i} \to PGL_{N_j}$ such that, for $g \in G_{geom}$, $\sigma(\pi_i(g))=\pi_j(g)$.

Were this true, the adjoint representation of $PGL_{N_i}$ and the adjoint representation of $PGL_{N_j}$ would be isomorphic as representations of $G_{geom}$. Because those representations are also self-dual, they would be dual to each other as representations of $G_{geom}$. Hence there would be a $G_{geom}$-invariant vector in the tensor product of the adjoint representation of $PGL_{N_i}$ with the adjoint representation of $PGL_{N_j}$. Since $V_i \otimes V_i^\vee$ is the trivial representation plus the adjoint representation of $PGL_{N_i}$, it follows that the invariant subspace of $V_i \otimes V_i^\vee\otimes V_j \otimes V_j^\vee$ would be at least two-dimensional. But this is not the case by Lemma \ref{indmoments}. \end{proof}

\begin{cor} Assume that $d \geq 4$. Then $G_{geom}$ contains $\prod_{i=1}^r SL_{N_i}$. \end{cor}

\begin{proof} Every algebraic group which projects surjectively to $\prod_{i=1}^r PGL_{N_i} $ must contain $\prod_{i=1}^r SL_{N_i}$. (Because $(f_1(x),\dots,f_r(x))$ and $(f_1'(x),\dots,f_r'(x))$  is perfect, the commutator subgroup of $G_{geom}$ also projects surjectively onto $\prod_{i=1}^r PGL_{N_i} $,  the commutator subgroup is a closed subset of $\prod_{i=1}^r SL_{N_i}$, and the projection $\prod_{i=1}^r SL_{N_i} \to \prod_{i=1}^r PGL_{N_i} $, so any closed subset with dense image is all of $\prod_{i=1}^r SL_{N_i}$).\end{proof}

\section{Proof of the main theorem in the case of middle extension sheaves}

\begin{proof}[Proof of Theorem \ref{main-geometric}]  Let $\G_1,\dots,\G_r$ be the sheaves constructed in Section \ref{sheaf-construction}. Let $N_i$ be the rank of $\G_i$ on the open subset where it is lisse. Let $U$ be the largest open subset where all are lisse. Define $S_{\mathbb F_q}$ as the set of $\Lambda$ such that $\Lambda_y$ lies in $ U(\mathbb F_q)$. By Lemma \ref{swan-lisse}, $U(\mathbb F_q) \geq \left( 1- \frac{ \sum_{i=1}^r m_i}{q} \right) q^d$.

Then the statement about the rank follows from Lemma \ref{shconstruction} and the statement about the purity follows from Lemma \ref{shweight}. The fact that the semisimplified conjugacy classes are conjugate to elements of the unitary group follows from the fact that they are semisimple by construction and their eignvalues all have norm one by purity. So it remains to prove 

\[ \lim_{q \to \infty}\left(  \frac{ \sum_{\Lambda \in S_q } f ( \varphi_{\Lambda,1},\dots, \varphi_{\Lambda,r}) } { \left| S_q \right|} - \frac{ \sum_{\Lambda \in S_q } \langle f \rangle ( \det \varphi_{\Lambda,1},\dots, \det \varphi_{\Lambda,r}) } { \left| S_q \right|}  \right) = 0  . \]

Because every term depends linearly on $f$, it suffices to prove this for a basis of the space of continuous conjugacy-invariant functions on $f$. In particular, we may assume that $f$ is the character of an irreducible representation of $\prod_{i=1}^{r} U (N_i)$, or, equivalently, the restriction to  $\prod_{i=1}^{r} U(N_i)$ of the character of an irreducible representation $V$ of  $\prod_{i=1}^{r} GL_{N_i}$. We divide into two cases depending on whether or not $V$ is one-dimensional.

If $V$ is one-dimensional, then it factors through the abelianization of $\prod_{i=1}^{r} GL_{N_i}$, so it depends only on the $r$ determinants. Hence $f= \langle f \rangle$ and the statement is trivial. 

If $V$ is not one-dimensional, then it remains nontrivial and irreducible as a representation of $\prod_{i=1}^r SU(N_i)$, and hence the average of the trace over any coset of  $\prod_{i=1}^r SU( N_i) $ vanishes, so $\langle f \rangle$ is zero. So it is sufficient to prove in this case that \[ \frac{ \sum_{\Lambda \in S_q }\tr ( \varphi_{\Lambda,1},\dots, \varphi_{\Lambda,r},V) } { \left| \left\{\Lambda \mid \Lambda  \in S_q \right\} \right|} =o(1) . \] In fact, we will show it is  $O \left( \frac{1}{\sqrt{q}} \right).$

Here we proceed as in the proof of Deligne's equidistribution theorem. We extend $V$ to an algebraic representation of $\prod_{i=1}^r GL_{N_i}$ (that again does not factor through the determinant map). We may compose $V$ with the monodromy representations of $\G_i$, $\pi_1(U) \to GL_{N_i}(\Ql)$ to obtain a representation of $\pi_1(U)$, and hence a lisse sheaf $\mathcal V$ on $U$. Let $e_i$ be such that the scalars in $GL_{N_i}$ act on $V$ by the character $\lambda \mapsto \lambda^{e_i}$. Then we have
\[= \tr( ( \varphi_{\Lambda,1},\dots, \varphi_{\Lambda,r}),V)=\tr \left( \left( \frac{\Frob_q}{q^{\frac{w_1+1}{2}}},\dots,  \frac{\Frob_q}{q^{\frac{w_r+1}{2}}}\right), V \right) \]
(Here the $i$th $\Frob_q$ is as an element of $GL_{N_i}$ by its action on $H^1_c(\mathbb A^1_{\Fq}, \F_i \otimes \L_{\Lambda})$   )
\[  =\frac{ \tr \left(\left( \Frob_q,\dots,\Frob_q\right) ,V\right)} {q^{ e_1\frac{w_1+1}{2} + \dots + e_r \frac{w_r+1}{2}}} = \frac{ \tr(\Frob_q, \mathcal V, \Lambda)} { q^{ e_1\frac{ w_1+1}{2} + \dots + e_r \frac{w_r+1}{2}}}.\]

Hence it suffices to show that  \[\sum_{x \in U(\mathbb F_q)} \tr(\Frob_q, \mathcal V_x) = O \left( q^{d - \frac{1}{2} + \sum_{i=1}^r e_i \frac{w_i+1}{2} }\right)\]

By the Lefschetz fixed point formula, this is $\sum_{i=0}^{2d} (-1)^i \tr(\Frob_q, H^i_c(U_{\Fq},\mathcal V))$. Because $\Frob_q$ acts on $V$ at a typical point of $U$ with eigenvalues of size $q^{ \sum_{i=1}^r e_i (w_i+1)/2}$, $\mathcal V$ is pure of weight $ \sum_{i=1}^r e_i (w_i+1)$, and so $H^i_c(U_{\Fq},\mathcal V))$ is mixed of weight $\leq i + \sum_{i=1}^r e_i (w_i+1)$ and hence \[\tr(\Frob_q, H^i_c(U_{\Fq},\mathcal V)) =O \left(q^{ i + \sum_{i=1}^r e_i \frac{ w_i+1}{2}}\right)= O \left( q^{d -\frac{1}{2} + \sum_{i=1}^r e_i \frac{w_i+1}{2})}\right)\] for $i\leq 2d-1$. Furthermore $H^{2d}_c(U_{\Fq},\mathcal V)$ vanishes because, as $V$ is irreducible and does not factor through the determinant, no subrepresentation or quotient representation of $V$ factors through the determinant, hence no quotient is invariant under the geometric monodromy group, i.e. $\mathcal V$ has no geometric monodromy invariants, and its top cohomology vanishes. 

Hence every term is $O\left( q^{d - \frac{1}{2} + \sum_{i=1}^r e_i \frac{w_i+1}{2}}\right)$, as desired.
\end{proof}
\bibliographystyle{alpha}

\bibliography{references}

\begin{thebibliography}{KRRGR18}

\bibitem[ABPS16]{llSLN}
A.~S. Aubert, Paul~D. Baum, Roger Plymen, and Maarten Solleveld.
\newblock The local {L}anglands correspondence for inner forms of {$SL_n$}.
\newblock {\em Research in the Mathematical Sciences}, 3(32), 2016.
\newblock https://doi.org/10.1186/s40687-016-0079-4.

\bibitem[Bry83]{brylinski}
Jean-Louis Brylinski.
\newblock Th\'{e}orie du corps de classes de {K}ato et rev\^{e}tements
  ab\'{e}liens de surfaces.
\newblock {\em Annales de l'institut Fourier}, 33(3):23--38, 1983.

\bibitem[Del80]{weilii}
Pierre Deligne.
\newblock La conjecture de {W}eil: {I}{I}.
\newblock {\em Publications math\'{e}matiques de l'I.H.\'{E}.S.}, 52:137--252,
  1980.

\bibitem[GT05]{guralnicktiep}
Robert~M. Guralnick and Pham~Huu Tiep.
\newblock Decompositions of small tensor powers and {L}arsen's conjecture.
\newblock {\em Representation Theory}, 9:138--208, 2005.
\newblock (electronic).

\bibitem[HKRG17]{hkr}
Chris Hall, Jon Keating, and Edva Roditty-Gershon.
\newblock Variance of sums in arithmetic progressions of arithmetic functions
  associated with higher degree {$L$}-functioins in {$\mathbb F_q[T]$}.
\newblock https://arxiv.org/abs/1703.09190, 2017.

\bibitem[Kat90]{esde}
Nicholas~M. Katz.
\newblock {\em Exponential Sums and Differential Equations}, volume 124 of {\em
  Annals of Mathematics Studies}.
\newblock Princeton University Press, 1990.

\bibitem[Kat95]{rls}
Nicholas~M. Katz.
\newblock {\em Rigid Local Systems}, volume 139 of {\em Annals of Mathematics
  Studies}.
\newblock Princeton University Press, 1995.

\bibitem[Kat05]{mmp}
Nicholas~M. Katz.
\newblock {\em Moments, Monodromy, and Perversity}, volume 159 of {\em Annals
  of Mathematics Studies}.
\newblock Princeton University Press, 2005.

\bibitem[Kat13a]{pdc}
Nicholas~M. Katz.
\newblock On a question of {K}eating and {R}udnick about primitive {D}irichlet
  characters with squarefree conductor.
\newblock {\em International Mathematics Research Notices},
  2013(14):3221--3249, 2013.

\bibitem[Kat13b]{wvqkr}
Nicholas~M. Katz.
\newblock Witt vectors and a question of {K}eating and {R}udnick.
\newblock {\em International Mathematics Research Notices}, 2013(15):3613?3638,
  2013.

\bibitem[KRRGR18]{krrr}
Jon Keating, Brad Rodgers, Edva Roditty-Gershon, and Zeev Rudnick.
\newblock Sums of divisor functions in {$\mathbb {F}_q[t]$} and matrix
  integrals.
\newblock {\em Mathematische Zeitschrift}, 288:167--198, 2018.

\bibitem[Laf02]{Lafforgue}
Laurent Lafforgue.
\newblock Chtoucas de {D}rinfeld et correspondance de {L}anglands.
\newblock {\em Inventiones mathematicae 147}, 1:1--242, 2002.

\bibitem[Lau81]{LaumonSMF}
G\'{e}rard Laumon.
\newblock Semi-continuit\'{e} du conducteur de {S}wan (d'apr\`{e}s {P}.
  {D}eligne).
\newblock {\em Ast\'{e}risque}, 83:173--219, 1981.

\end{thebibliography}

\end{document}